\documentclass{amsart}
\usepackage[dvipdfmx]{graphicx}
\usepackage{color}
\usepackage{overpic}
\usepackage{lineno}
\usepackage[all]{xy}
\usepackage{amssymb,amsmath}
\usepackage{tikz-cd}

\newtheorem{theorem}{Theorem}[section]
\newtheorem{corollary}[theorem]{Corollary}
\newtheorem{proposition}[theorem]{Proposition}
\newtheorem{lemma}[theorem]{Lemma}

\newtheorem*{thmintro}{Theorem~\ref{thm:main}}
\newtheorem*{corintro}{Corollary~\ref{cor:main}}

\theoremstyle{definition}

\newtheorem{question}[theorem]{Question}

\theoremstyle{remark}
\newtheorem{remark}[theorem]{Remark}

\definecolor{MyGreen}{rgb}{0, 0.5, 0}
\definecolor{MyYellow}{rgb}{1.0, 1.0, 0}
\definecolor{MyGray}{rgb}{0.5, 0.5, 0.5}
\definecolor{MyBlue}{rgb}{0, 0.4, 0.7}
\definecolor{MyOrange}{rgb}{1,.576470588,.117647059}
\definecolor{MyPink}{rgb}{1,0,1}
\definecolor{Crimson}{rgb}{0.7, 0.2, 0.2}
\definecolor{PG}{cmyk}{0.2, 0, 0.2, 0}
\definecolor{PM}{cmyk}{0, 0.2, 0, 0}

\newcommand{\Red}[1]{{\color{red}#1}}
\newcommand{\Blue}[1]{{\color{blue}#1}}
\newcommand{\Green}[1]{{\color{green}#1}}
\newcommand{\Gray}[1]{{\color{MyGray}#1}}

\newcommand{\Orange}[1]{{\color{MyOrange}#1}}

\newcommand{\Magenta}[1]{{\color{MyPink}#1}}

\def\Z{\mathbb{Z}}
\def\C{\mathbb{C}}
\author{Kazuhiro Ichihara}
\address{Department of Mathematics, College of Humanities and Sciences, Nihon University, 3-25-40 Sakurajosui, Setagaya-ku, Tokyo 156-8550, Japan.}
\email{ichihara@math.chs.nihon-u.ac.jp}
\thanks{Ichihara is partially supported by Grant-in-Aid for Scientific Research (C) 
(No.\ 26400100), Japan Society for the Promotion of Science.}

\author{In Dae Jong}
\address{Department of Mathematics, Kindai University, 3-4-1 Kowakae, Higashiosaka City, Osaka 577-0818, Japan}
\email{jong@math.kindai.ac.jp}

\author{Kouki Taniyama}
\address{Department of Mathematics, School of Education, Waseda University, Nishi-Waseda 1-6-1, Shinjuku-ku, Tokyo, 169-8050, Japan}
\email{taniyama@waseda.jp}
\thanks{Taniyama is partially supported by Grant-in-Aid for Challenging Exploratory Research (No.\ 15K13439) and Grant-in-Aid for Scientific Research (A) (No.\ 16H02145), Japan Society for the Promotion of Science.}



\title[Achiral hyperbolic 3-manifolds]
{Achiral 1-cusped hyperbolic 3-manifolds not coming from 
amphicheiral null-homologous knot complements}

\begin{document}

%\pagewiselinenumbers

\begin{abstract} 
It is experimentally known that achiral hyperbolic 3-manifolds are quite sporadic at least among those with small volume, 
while we can find plenty of them as amphicheiral knot complements in the 3-sphere. 
In this paper, we show that there exist infinitely many achiral 1-cusped hyperbolic 3-manifolds 
not homeomorphic to any amphicheiral null-homologous knot complement in any closed achiral 3-manifold. 
%Each of them is obtained from a chirally cosmetic banding on a certain 2-bridge link by applying the reverse of the Montesinos trick. 
\end{abstract}

\maketitle

\section{Introduction}\label{sec:intro}

An oriented $3$-manifold is said to be \emph{achiral}\footnote{The term ``amphicheiral'' 
is also used for the same notion on a 3-manifold. 
In this paper, we use the term ``achiral'' for a 3-manifold 
to clearly distinguish from an ``amphicheiral knot''.} 
if it admits an orientation-reversing self-homeomorphism. 
As pointed out in~\cite{MartelliPetronio}, it is experimentally known that achiral hyperbolic 3-manifolds are quite sporadic at least among those with small volume. 
Similarly, the achiral cusped hyperbolic 3-manifolds obtained by Dehn fillings on the ``magic manifold'' are quite rare; 
just the figure-eight knot complement, its sibling, and the complement of the two-bridge link $S(10,3)$~\cite[Proposition 2.9]{MartelliPetronio}. 
For the definition of a Dehn filling, see Section~\ref{sec:preliminary}. 

Among them, the figure-eight sibling\footnote{It is a unique complete orientable hyperbolic 3-manifold constructed by gluing together two ideal regular tetrahedra such that it is not homeomorphic to the figure-eight knot complement~\cite{Weeks}.} 
is much more special. 
%since the pair of slopes (i.e., isotopy classes of non-trivial simple closed curves \fbox{on a peripheral torus}) 
%corresponding to $1/0$ and $-1/1$ are equivalent 
%via an orientation-reversing self-homeomorphism. 
%the $1/0$- and the $-1/1$-Dehn fillings yield orientation-reversingly homeomorphic manifolds, 
%%as pointed out first by Weeks in his thetis~\cite{Weeks}, 
%that is to say, the pair of Dehn fillings are \emph{chirally cosmetic} (see Section~\ref{sec:preliminary} for the precise definitions).  
%%(with respect to some meridian-longitude system). 
%In other words, 
%Precisely, there exists an orientation-reversing self-homeomorphism $h$ 
%on the figure-eight sibling such that 
%the distance between a slope $\gamma$ and the slope $h(\gamma)$ is just one. 
Precisely, the figure-eight sibling admits an orientation-reversing self-homeomorphism $h$ 
for which there exists a slope $\gamma$ on a horotorus such that 
the distance between $\gamma$ and its image under $h$ is just one. 
Here a \emph{slope} is an isotopy class of a non-trivial simple closed curve on a torus, 
and the \emph{distance} of two slopes is defined as the minimal intersection number 
between their representatives. 

We note that such a phenomenon does not occur for any amphicheiral null-homologous knot complement. 
Here a knot $K$ in an achiral 3-manifold $M$ 
with an orientation-reversing self-homeomorphism $\varphi$ 
is said to be \emph{amphicheiral (with respect to $\varphi$)} 
if $K$ is isotopic to $\varphi(K)$ in $M$. 
For an amphicheiral null-homologous knot $K$, 
with respect to the meridian-preferred longitude system for $K$, 
the slope $p/q$ changed to $-p/q$ by the restriction of $\varphi$ 
to the complement of $K$ 
since $\varphi$ preserves the meridian-preferred longitude system for $K$. 
%\fbox{Here  $E(K) = M \setminus N(K)$ and $N(K)$} 
%\fbox{denotes an open tubular neighborhood of $K$. }
Note that the distance between the slopes $p/q$ and $-p/q$ is $2|pq|$, 
in particular, it is even. 
In addition, if $|p|, |q| \ne 0$, then the distance is greater than one. 
%It follows that no more examples can be created from amphicheiral knots in $S^3$. 
%Hence one's intuition might say that cosmetic fillings along slopes with distance one must be rare. 

In view of this, it is natural to ask if there exist achiral 3-manifolds not coming from any amphicheiral null-homologous knot complements (other than the figure-eight sibling). 
As the main result in this paper, the following ensures that such 3-manifolds surely exist. 

\begin{thmintro}
There exist infinitely many achiral 1-cusped hyperbolic 3-manifolds 
each of which is not homeomorphic to any amphicheiral null-homologous knot complement in any closed achiral 3-manifold. 
\end{thmintro}

One of our motivations to study such 3-manifolds comes from the research 
about cosmetic surgeries on knots in \cite{IJM}. 
There the achirality of the figure-eight sibling can be observed, 
and the 3-manifolds in Theorem~\ref{thm:main} are considered as extensions to it. 
In fact, the first two authors asked the following in \cite[Question 4.1]{IJM}. 

\begin{question}
Can we find chirally cosmetic Dehn fillings on an achiral cusped hyperbolic 3-manifold 
along distance one slopes (other than the figure-eight sibling)? 
\end{question}

The 3-manifolds in Theorem~\ref{thm:main} give 
an affirmative answer to the above question as follows. 

\begin{corintro}
There exist infinitely many achiral 1-cusped hyperbolic 3-manifolds 
each of which admits chirally cosmetic Dehn fillings along distance one slopes. 
\end{corintro}

%The infinite family of manifolds in Theorem~\ref{thm:main} and Corollary~\ref{cor:main} are same. 
We construct our 3-manifolds in Section~\ref{sec:examples}, 
which are given as the double branched cover of certain 2-string tangles. 
Then, after recalling the definitions of a Dehn surgery, a Dehn filling, a banding, and the Montesinos trick in Section~\ref{sec:preliminary}, 
%and introduce terminologies used in this paper. 
%We also 
%In particular, we recall a fact that a banding is closely related to 
%an integral Dehn surgery by considering the double branched covering space. 
we show in Section~\ref{sec:hyp} that the interior of the obtained 3-manifolds are hyperbolic by using the Montesinos trick. 
%we obtain Theorem~\ref{thm:main} and Corollary~\ref{cor:main}.  
%In Section~\ref{sec:homology}, we calculate the homology groups of our manifolds, 
%which guarantee that each of our examples is not homeomorphic to any null-homologous knot complement in any closed oriented 3-manifold. 
In the last section, we show that each of our 3-manifolds is realized as 
the exterior of an amphicheiral knot in some achiral 3-manifold.

\section{Examples}\label{sec:examples}

In this section, we construct achiral 3-manifolds %with a torus boundary 
by taking the double branched covers over certain ``amphicheiral'' 2-string tangles. 
Here, by a \emph{2-string tangle} or simply a \emph{tangle}, 
we mean a pair consisting of a 3-ball and two properly embedded arcs in the 3-ball. 

%%%%%%%%%%%%%%%%%%%%%%%%%%%%%%%%
%%% Delete T_n^\pm and treat only T_n (170912) %%%% 
%%%%%%%%%%%%%%%%%%%%%%%%%%%%%%%%
%For any integer $n$, let us consider the tangles $T^{\pm}_n$ depicted in Figure~\ref{fig_Tn}. 
%Here the box labelled $n$ (resp.\ $-n$) indicates $n$ times right-handed  (resp.\ left-handed) half twists (see Figure~\ref{fig_twists}). 
%
%One can see that 
%$T^{+}_{-n}$ is obtained from $T^{-}_n$ 
%by performing crossing changes at all the crossings, 
%that is, $T^{+}_{-n}$ is the mirror image of $T^-_{n}$, 
%and also is obtained by a $\pi/2$-rotation from $T^-_n$. 
%
%\begin{figure}[!htb]
%\centering
%\begin{overpic}[width=.9\textwidth]{fig_Tn.eps}
%\put(9.6,29.3){$n$} 
%\put(8.3,9.7){$-n$} 
%\put(29.5,9.5){$n$} 
%\put(28,29.5){$-n$} 
%\put(68.8,29.3){$n$} 
%\put(67.5,9.7){$-n$} 
%\put(88.5,9.5){$n$} 
%\put(87,29.5){$-n$} 
%\end{overpic}
%\caption{The tangles $T_n^-$ and $T_{n}^+$. }
%\label{fig_Tn}
%\end{figure}

Let $n$ be an integer with $n \ne 0,1$. 
Unless otherwise noted, we assume that $n \ne 0, 1$ throughout this paper. 
Let us consider the tangle $T_n$ depicted in Figure~\ref{fig_Tn}. 
Here the box labelled $n$ (resp.\ $-n$) indicates $n$ times right-handed (resp.\ left-handed) half twists (see Figure~\ref{fig_twists}). 
The construction of $T_n$ is based on 
the two-component link %$L(m, \varepsilon_1, \varepsilon_2)$ 
proposed by Nikkuni and the third author in \cite[Figure 1.2]{NikkuniTaniyama}. 

\begin{figure}[!htb]
\centering
\begin{overpic}[width=.35\textwidth]{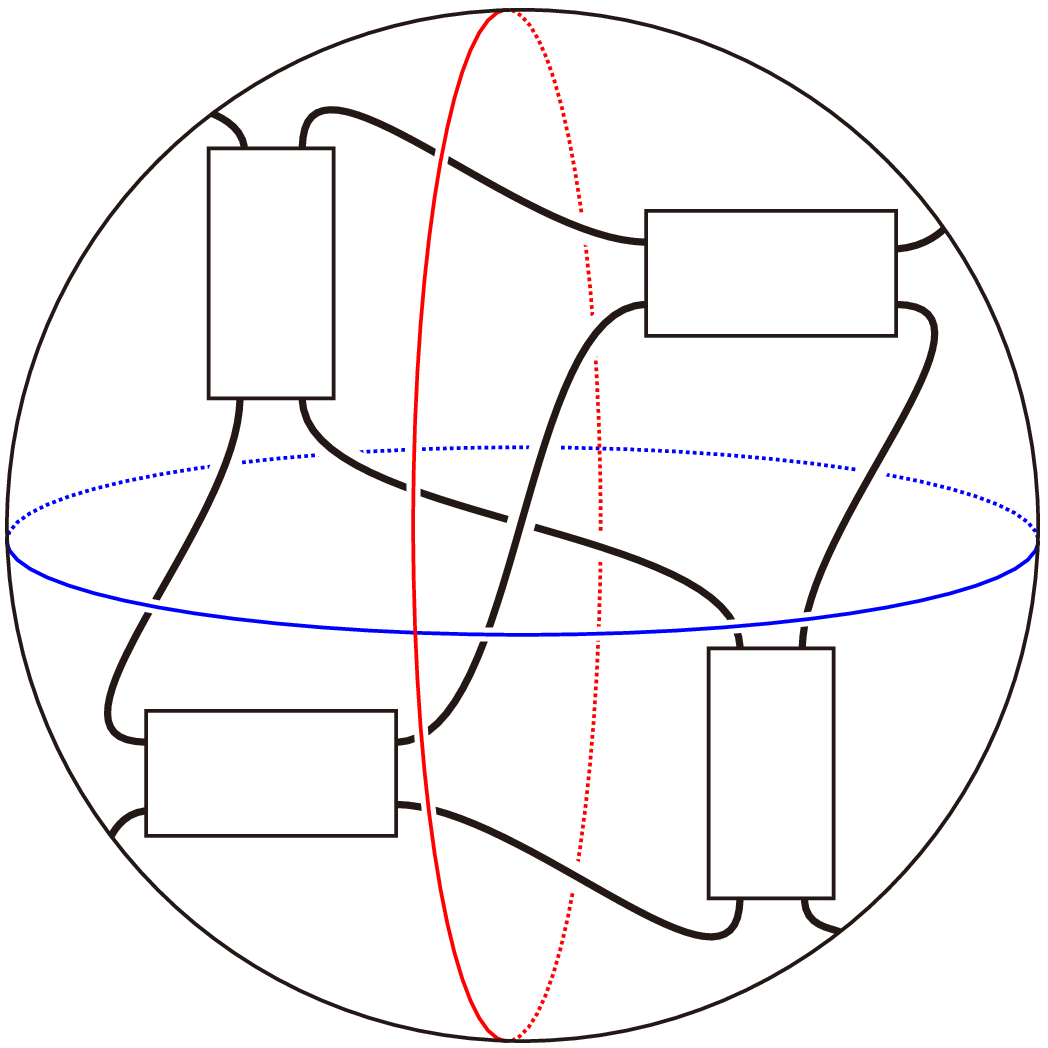}
\put(24,72){$n$} 
\put(20,24){$-n$} 
\put(68,72){$-n$} 
\put(72,23.5){$n$} 
\put(47,102){\Red{$\mu$}}
\put(101,47){\Blue{$\lambda$}}
\end{overpic}
\caption{The tangle $T_n$. } 
%The red (resp.\ blue) curve represents meridian $\mu$ (resp.\ the longitude $\lambda$) for $T_n$.
\label{fig_Tn}
\end{figure}

\begin{figure}[!htb]
\centering
\begin{overpic}[width=.7\textwidth]{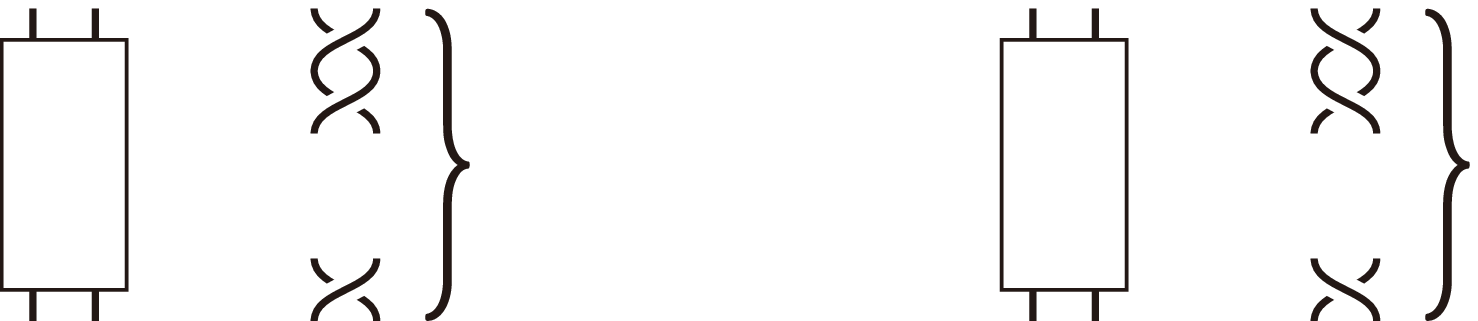}
\put(3.2,9){$n$}
\put(13,9){$=$}
\put(22.7,6.3){\rotatebox{90}{$\cdots$}}
\put(33,14){$n$ times }
\put(33,10){right-handed}
\put(33,6){half twists}
\put(69.5,9){$-n$}
\put(80.8,9){$=$}
\put(90.7,6.3){\rotatebox{90}{$\cdots$}}
\put(101,14){$n$ times}
\put(101,10){left-handed}
\put(101,6){half twists}
\end{overpic}
\caption{The cases where $n$ is positive (left-hand side) and negative (right-hand side).}
\label{fig_twists}
\end{figure}

Let $M_n$ be the compact oriented 3-manifold with a torus boundary 
obtained as %the interior of 
the double branched cover over $T_n$. 
Then we have the following. 

\begin{proposition}\label{prop:ampmfd} 
The 3-manifold $M_n$ is achiral, 
and is not homeomorphic to any amphicheiral null-homologous knot exterior
in any closed achiral 3-manifold.
\end{proposition}

Here for a knot $K$ in a 3-manifold $M$, 
we denote by $N(K)$ an open tubular neighborhood of $K$, 
and the \emph{exterior}, denoted by $E(K)$, of $K$ is defined as $M \setminus N(K)$. 
%:
Note that the interior of $E(K)$ is homeomorphic to the complement, $M \setminus K$, of $K$.

\begin{proof}[Proof of Proposition~\ref{prop:ampmfd}] 
Let $m$ be the map mirroring $T_n$ along the plane of the paper, and let $T_n! = m(T_n)$. 
Notice that $T_n!$ is the tangle obtained from $T_n$ by performing crossing 
changes at all the crossings. 
Let $M_n!$ be the oriented 3-manifold obtained as %the interior of 
the double branched cover over $T_n!$. 
Then the lift $\widetilde m \colon M_n \to M_n!$ %of $m \colon T_n \to T_n!$ 
is an orientation-reversing homeomorphism. 
On the other hand, since $T_n$ is obtained by a $\pi/2$-rotation from $T_n!$, 
%say $r \colon T_n \to T_n!$, 
$M_n!$ is orientation-preservingly homeomorphic to $M_n$. 
Therefore %$M_n$ and $-M_n$ are orientation-preservingly homeomorphic, that is, 
$M_n$ admits an orientation-reversing self-homeomorphism $h = \widetilde r \circ \widetilde m$, 
where $\widetilde r \colon M_{n}! \to M_{n}$ is the lift of the $\pi/2$-rotation 
$r \colon T_{n}! \to T_n$. 
This implies that $M_n$ is achiral. 

On the torus boundary $\partial M_n$ of $M_n$, 
we consider the two slopes $\widetilde\mu$ and $\widetilde\lambda$ 
appearing as the preimage of the two loops $\mu$ and $\lambda$ % on $\partial T_n$ 
depicted in Figure~\ref{fig_Tn} respectively. 
Then the distance between $\widetilde\mu$ and $\widetilde\lambda$ is just one. 
Since $r \circ m (\mu) = \lambda$, 
we have $h(\widetilde\mu) = \widetilde\lambda$. 
Then, as explained in Section~\ref{sec:intro}, 
% the existence of the self-homeomorphism $h$ implies that 
$M_n$ is not homeomorphic to any amphicheiral null-homologous knot exterior 
in any closed achiral 3-manifold. 
\end{proof}

%One can see that $T_n!$ is obtained by a $\pi/2$-rotation from $T_n$. 

%\begin{figure}[!htb]
%\centering
%\begin{overpic}[grid,width=.35\textwidth]{fig_Tn.eps}
%\put(9.6,29.3){$n$} 
%\put(8.3,9.7){$-n$} 
%\put(29.5,9.5){$n$} 
%\put(28,29.5){$-n$} 
%\end{overpic}
%\caption{The tangles $T_n$. }
%\label{fig_Tn}
%\end{figure}

\section{Dehn surgery, banding, and Montesinos trick}\label{sec:preliminary}

In this section, we give a review on a Dehn surgery, a Dehn filling, a banding, and the Montesinos trick, to show that the interior of $M_n$ is hyperbolic for $n \ne 0,1$. 

\subsection{Dehn surgery}

Let $K$ be a knot in a closed oriented 3-manifold $M$ with the exterior $E(K)$. 
Let $\gamma$ be a slope on the boundary torus $\partial E(K)$. 
Then, the \textit{Dehn surgery on $K$ along $\gamma$} 
is defined as the following operation: 
Glue a solid torus $V$ to $E(K)$ such that a simple closed curve representing $\gamma$ bounds a meridian disk in $V$. 
We denote by $K(\gamma)$ the obtained 3-manifold. 
It is said that the Dehn surgery along the meridional slope 
is the \textit{trivial} Dehn surgery. 
Also it is said that a Dehn surgery along 
a slope which is represented by a simple closed curve 
intersecting the meridian at a single point 
is an \textit{integral} Dehn surgery. 

In the case where $K$ is null-homologous in $M$, 
we have the well-known bijective correspondence 
between $\mathbb{Q} \cup \{ 1/0 \}$ and the set of slopes on $\partial E(K)$, 
which is given by using the meridian-preferred longitude system for $K$. 
When the slope $\gamma$ corresponds to $r \in \mathbb{Q} \cup \{1/0\}$, then the Dehn surgery on $K$ along $\gamma$ is said to be the \textit{$r$-surgery} on $K$, 
and the obtained 3-manifold is denoted by $K(r)$.
In this case, an integral Dehn surgery corresponds to an $n$-surgery with an integer $n$. 

Two Dehn surgeries on a knot $K$ along 
two slopes are said to be \emph{chirally cosmetic}\footnote{This terminology was used in Kirby's problem list \cite{Kirby}. 
In \cite{BleilerHodgsonWeeks}, it is called \emph{reflectively cosmetic}.} 
if two obtained 3-manifolds are orientation-reversingly homeomorphic.

\subsection{Dehn filling}

%A Dehn filling is a simple extension of a Dehn surgery. 
% starting a 3-manifold with a torus boundary. 
While there are some overlaps with notions on Dehn surgery, 
here we briefly review a Dehn filling. 
Let $M$ be a compact connected oriented 3-manifold with a torus boundary $\partial M$, 
and $\gamma$ a slope on $\partial M$. 
The \emph{Dehn filling on $M$ along $\gamma$} 
is the operation gluing a solid torus $V$ to $M$ so that 
a simple closed curve representing $\gamma$ bounds a meridian disk in $V$. 
%We denote the obtained manifold by $M(\gamma)$. 
%In the case where $M$ is the exterior 
%(i.e.\ the complement of an open tubular neighborhood) 
%of a null-homologous knot in a closed oriented 3-manifold, 
%we have the well-known bijective correspondence 
%between $\mathbb{Q} \cup \{ 1/0 \}$ and the set of slopes on $\partial M$, 
%which is given by using the meridian-preferred longitude system for $K$. 
%When the slope $\gamma$ corresponds to $r \in \mathbb{Q} \cup \{1/0\}$, 
%then the Dehn filling on $M$ along $\gamma$ is said to be 
%the \textit{$r$-Dehn filling} on $M$. 
As in the case for a Dehn surgery, 
if we have the correspondence 
between the slope $\gamma$ and $r \in \mathbb{Q} \cup \{1/0\}$ 
using the meridian-preferred longitude system, 
then the Dehn filling on $M$ along $\gamma$ is said to be the \textit{$r$-Dehn filling} on $M$. 

Two Dehn fillings on $M$ along 
two slopes are said to be \emph{chirally cosmetic} if 
the resultants are orientation-reversingly homeomorphic. 
For recent studies on chirally cosmetic Dehn fillings, 
we refer the reader to \cite{IJM, IchiharaSaito, IchiharaWu, NiWu}.

%\subsection{Dehn filling}
%
%The \emph{trivial} Dehn filling, an \emph{integral} Dehn filling, 
%a \emph{purely} or \emph{chirally cosmetic} Dehn filling are defined by the same way as these of 
%Dehn surgeries. 

\subsection{Banding}
We call the following operation on a link a \textit{banding}\footnote{The operation is sometimes called a band surgery, a bund sum (operation), or a hyperbolic transformation in a variety of contexts. 
In this paper, referring to \cite{Bleiler}, we use the term banding to clearly distinguish it from a Dehn surgery on a knot.} on the link. 
For a given link $L$ in $S^3$ and an embedding $b \colon I \times I \to S^3$ such that $b ( I \times I ) \cap L = b ( I \times \partial I )$, where $I$ denotes a closed interval, we obtain a (new) link as $ ( L - b ( I \times \partial I ) ) \cup b ( \partial I \times I )$. 
%We call the link so obtained \emph {the link obtained from $L$ by a banding 
%along the band $b$}. 

\begin{remark}
On performing a banding, it is often assumed the compatibility of orientations of the original link and the obtained link, but in this paper, we do not assume that. 
Also note that this operation for a knot yielding a knot appears as the $n=2$ case of the $H(n)$-move on a knot, which was introduced in \cite{HosteNakanishiTaniyama}. 
\end{remark}

%In the next section, we will consider rational tangles. 
It is well-known that a rational tangle is determined by the meridional disk in the tangle. 
The boundary curve of the meridional disk is 
parameterized by an element of $\mathbb{Q} \cup \{1/0\}$, 
called a \emph{slope} of the rational tangle. 
A rational tangle is said to be \emph{integral} if the slope is an integer or $1/0$. 
For brevity, we call an integral tangle with a slope $n$ an \emph{$n$-tangle}. 

A banding can be regarded as an operation replacing a $1/0$-tangle into an $n$-tangle. 
Then we call this banding an \emph{$n$-banding}. 

A banding on a link $L$ is said to be \emph{chirally cosmetic} 
if the link obtained from $L$ by the banding is ambient isotopic to 
the mirror image $L!$ of $L$ in $S^3$.

\subsection{Montesinos trick}

We here recall the Montesinos trick originally introduced in \cite{Montesinos}. 
Let $\Sigma$ be the double branched cover of $S^3$ branched along a link $L \subset S^3$. 
Let $K$ be a knot in $\Sigma$, which is \textit{strongly invertible} with respect to the preimage $\bar{L}$ of $L$, that is, there is an orientation preserving involution of $\Sigma$ with the quotient $S^3$ and the fixed point set $\bar{L}$ which induces an involution of $K$ with two fixed points. 
Then the 3-manifold $K(\gamma)$ obtained by an integral Dehn surgery on $K$ is homeomorphic to the double branched cover $\Sigma'$ along the link $L'$ obtained from $L$ by a banding along the band appearing as the quotient of $K$. 
%with the corresponding framing \fbox{slope?}. 
%See \cite{BakerBuck, BakerBuckLecuona, Bleiler} for example. 
That is, we have the following commutative diagram: 
\[\xymatrix@!C{
\bar L \subset \Sigma \ar[d]_{\text{double branched covering}} \qquad \ar[r]^{\text{Dehn surgery on }K} & \qquad \Sigma' = K(\gamma) \ar[d]^{\text{double branched covering}} \\
L \subset S^3 \qquad
\ar[r]^{\text{banding on }L} & \qquad L' \subset S^3 
}\]

\section{Hyperbolicity}\label{sec:hyp}

In this section, 
we show that the interior of our 3-manifolds $M_{n}$ given in Section~\ref{sec:examples} are hyperbolic, 
and prove our main theorem. 
Let $K_n$ be the link in $S^3$ obtained by closing the tangle $T_n$ 
as shown in Figure~\ref{fig_Kn}. 

\begin{figure}[!htb]
\centering
\begin{overpic}[width=.3\textwidth]{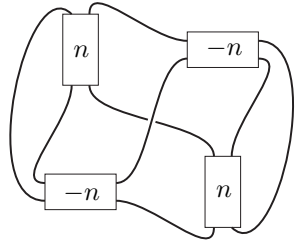}
\put(22.5,64){$n$}
\put(72.5,15){$n$}
\put(69,64.5){$-n$}
\put(19,14){$-n$}
%\put(64,6.5){$-n$}
%\put(65.5,29){$n$}
%\put(86.5,29){$-n$}
%\put(88,7){$n$}
\end{overpic}
\caption{The link $K_n$.}
%\caption{The knot $K_n^-$ (left-hand side) and the knot $K_n^+$ (right-hand side)}
\label{fig_Kn}
\end{figure}

\begin{proposition}\label{prop:type} 
The knot $K_n$ is the two-bridge knot (or link) with Schubert's normal form 
\[ S(n^4 - 2n^3 + 2n^2 - 2n +1, n^3 - 2n^2 + n - 1) \, . \] 
\end{proposition}
For the definitions of Schubert's normal form and Conway's normal form 
for a two-bridge link, see for example \cite[Section 2.1]{KawauchiSurvey}.  
\begin{proof}[Proof of Proposition~\ref{prop:type}] 
One can diagrammatically check that $K_n$ is the two-bridge knot (or link) 
with Conway's normal form $C(n, n, -1, n, n)$. 
Calculating the continued fraction, we have  
\[ 
\dfrac{1}{n + \dfrac{1}{n + \dfrac{1}{ -1 +  \dfrac{1}{ n + \dfrac{1}{n}}}}} 
= \dfrac{n^3 - 2n^2 + n - 1}{n^4 - 2n^3 + 2n^2 - 2n +1}\, .  \qedhere \] 
\end{proof}
We note that $K_n$ is a two-bridge knot when $n$ is even, 
and a two-bridge link when $n$ is odd. 
Also note that $K_2 = S(5,1)$ is the $(2,5)$-torus knot. 
%and $K^{+}_2 = S(45,19)$ is the knot $9_{23}$ in the knot table. 
%
%One can see that $K^-_{n}$ is the mirror image of $K^+_{-n}$, 
%for the latter is obtained by performing crossing changes at all the crossings of the former. 

\begin{figure}[!htb]
\centering
\begin{overpic}[width=\textwidth]{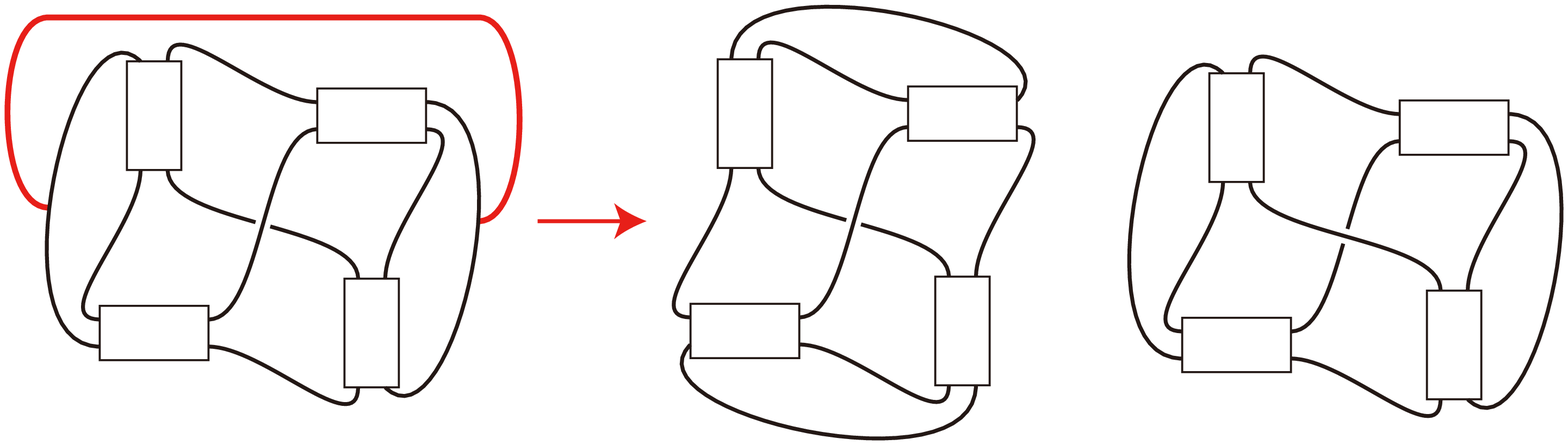}
\put(15,28){$\Red{0}$} 
\put(67,12){$=$} 
\put(7.5,6){$-n$} 
\put(22.8,6.5){$n$} 
\put(9,20){$n$} 
\put(21.3,20){$-n$} 
\put(45.5,6.5){$-n$} 
\put(60.8,6.5){$n$} 
\put(46.7,20.5){$n$} 
\put(59.3,20.5){$-n$} 
\put(78.5,5.5){$n$} 
\put(90.8,5.5){$-n$} 
\put(77,19.5){$-n$} 
\put(92.3,19.5){$n$} 
\end{overpic}
\caption{The $0$-banding on $K_n$ yields the mirror image $K_n!$. }
\label{fig_CosmeticBanding}
\end{figure}

As shown in Figure~\ref{fig_CosmeticBanding}, 
each of $K_n$ admits a chirally cosmetic banding. 
In particular, for $K_2$, this chirally cosmetic banding 
was essentially discovered by Zekovi\'c~\cite{Zekovic}, 
and pointed out by the first two authors that the upstairs of this banding corresponds to 
the chirally cosmetic filling on the figure-eight sibling, see \cite[Section 4]{IJM}. 

Let $\Sigma_n$ be the lens space of type 
$(n^4 - 2n^3 + 2n^2 - 2n +1, n^3 - 2n^2 + n - 1)$ 
which is obtained as the double branched cover of $S^3$ branched along $K_n$. 
Since $n \ne 0,1$, $\Sigma_n$ is not homeomorphic to $S^{3}$ or $S^{2} \times S^{1}$.  
Then we obtain the following. 
%Applying the Montesinos trick conversely to the cosmetic banding on $K_n$, 
%we obtain the following. 

\begin{theorem}\label{thm:main} 
There exist infinitely many achiral 1-cusped hyperbolic 3-manifolds 
each of which is not homeomorphic to any amphicheiral null-homologous knot complement in any closed achiral 3-manifold. 
\end{theorem}
\begin{proof} 
%We only consider the knot $K_n^-$ since the same argument can be applied to the knot $K_n^+$. 
%Set $K_n = K_n^-$, and let $\Sigma_n$ be the lens space of type 
Applying the isotopic deformation and the Montesinos trick conversely 
to the chirally cosmetic banding on $K_n$ shown in Figure~\ref{fig_CosmeticBanding}, 
by Proposition~\ref{prop:type}, 
we obtain a surgery descriptions of $\Sigma_n$ and $-\Sigma_n$ 
as in Figure~\ref{fig_BranchedCover}. 
That is, the trivial Dehn surgery and the 0-surgery on the red component in Figure~\ref{fig_BranchedCover} 
yield $\Sigma_n$ and $-\Sigma_n$ respectively. 

Drilling along the red component, 
we obtain the surgery description of the 3-manifold with a torus boundary $M_n$ 
as in Figure~\ref{fig_Mn}, 
which is already constructed in Section~\ref{sec:examples}. 
Then $M_n$ is the exterior of a knot in $\Sigma_n$, 
which admits chirally cosmetic Dehn fillings along distance one slopes. 
%Here note that $M_n$ is achiral by the construction. 

By the classification of cosmetic surgeries on a non-hyperbolic knot in a lens space 
due to Matignon~\cite{Matignon}, we see that the interior of $M_n$ is hyperbolic. 
In fact, for a non-hyperbolic knot $J$ in a lens space other than $S^{3}$ or $S^{2}\times S^{1}$, 
if the trivial Dehn surgery and the $r$-surgery on $J$ are chirally cosmetic, 
then $r \ne 0$, see \cite[Theorems 3.2 and 4.1]{Matignon}. 

It is enough to show that if $n \ne n'$, then $M_n$ is not homeomorphic to $M_{n'}$. 
Then the following lemma completes the proof. 
%By the classification of lens spaces (see \cite[Theorem D.1.2]{KawauchiSurvey} for example), 
%we see that if $n \ne n'$, then $\Sigma_n$ is not homeomorphic to $\Sigma_{n'}$. 
%Therefore if $n \ne n'$, then $M_n$ is not homeomorphic to $M_{n'}$. 
%\fbox{we have to use invariant (homology in the next section)?}
\end{proof}

\begin{figure}[!htb]
\centering
\begin{overpic}[width=.9\textwidth]{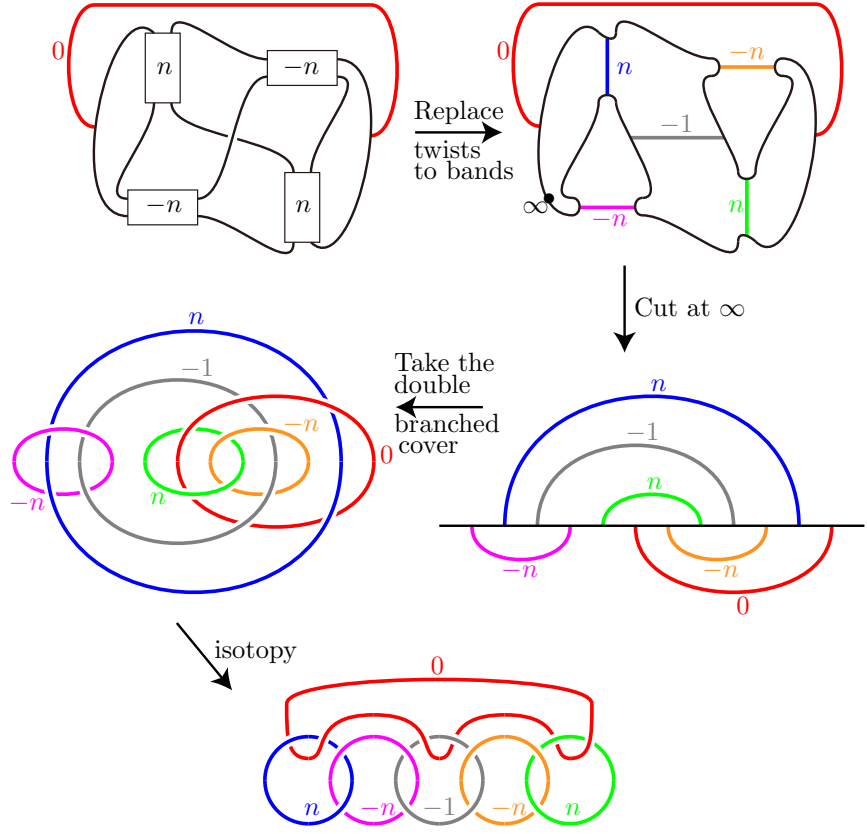}
\put(5,90){$\Red{0}$} 
\put(17.3,88.3){$n$} 
\put(16,72.1){$-n$} 
\put(32,88.3){$-n$} 
\put(33.5,72.1){$n$} 
\put(57,90){$\Red{0}$} 
\put(71,88.3){$\Blue{n}$} 
\put(68,70.5){$\Magenta{-n}$} 
\put(84,90){$\Orange{-n}$} 
\put(84,72.1){$\Green{n}$} 
\put(76,81.5){$\Gray{-1}$} 
\put(85,25.2){$\Red{0}$} 
\put(75,51.2){$\Blue{n}$} 
\put(57.5,29.5){$\Magenta{-n}$} 
\put(81,29.5){$\Orange{-n}$} 
\put(74.5,39.7){$\Green{n}$} 
\put(72,45.5){$\Gray{-1}$} 
\put(43.4,42.8){$\Red{0}$} 
\put(21,59){$\Blue{n}$} 
\put(0,37.3){$\Magenta{-n}$} 
\put(32,46.6){$\Orange{-n}$} 
\put(16.5,38){$\Green{n}$} 
\put(20,53){$\Gray{-1}$} 
\put(49.3,18.2){$\Red{0}$} 
\put(34.5,1.7){$\Blue{n}$} 
\put(41,1.7){$\Magenta{-n}$} 
\put(56.3,1.7){$\Orange{-n}$} 
\put(65,1.7){$\Green{n}$} 
\put(48.3,1.7){$\Gray{-1}$} 
\put(47.3,83){Replace} 
\put(47.3,78.5){twists} 
\put(47.3,76){to bands}  
\put(45,53.5){Take the} 
\put(45,51){double} 
\put(45,46.5){branched} 
\put(45.5,44){cover} 
\put(62.3,73){$\bullet$}
\put(60,72){$\infty$}
\put(73,60){Cut at $\infty$}
\put(24,20){isotopy} 
\end{overpic}
\caption{Applying the Montesinos trick.}
\label{fig_BranchedCover}
\end{figure}

\begin{lemma}\label{lem:homology} 
We have $H_1(M_n; \Z) \cong \Z \oplus \Z_{n^3 - n^2 + n -1}$.  
\end{lemma}
\begin{proof} 
Let $X$ be the exterior of the 6-components link in $S^{3}$ as in Figure~\ref{fig_generator}. 
Taking representatives of meridians $x, a, b, c, d, e$ as in Figure~\ref{fig_generator}, 
we have 
\[ H_1(X; \Z) \cong \Z^6 = \langle a \rangle \oplus \langle b \rangle \oplus \langle c \rangle \oplus \langle d \rangle \oplus \langle e \rangle \oplus \langle x \rangle\, . \]  
Reading off the slopes from Figure~\ref{fig_Mn}, 
we see that suitable Dehn fillings on the five boundaries supply the following relations.  
\begin{align*} 
\begin{cases} 
n a + x - b = 0 \, ,  \\ 
-na + c - a = 0 \, ,  \\ 
-c + b - x - d = 0 \, ,  \\ 
-nd -c + e = 0 \, ,  \\ 
ne + d + x = 0 \, . 
\end{cases}
\end{align*}
Then we have 
\[ 
\begin{pmatrix} 
n & -1 & 0 & 0 & 0 & 1 \\ 
-1 & n & 1 & 0 & 0 & 0 \\ 
0 & 1 & -1 & -1 & 0 & -1 \\ 
0 & 0 & -1 & -n & 1 & 0 \\ 
0 & 0 & 0 & 1 & n & 1 \\ 
\end{pmatrix}
\] 
as the presentation matrix of $H_1(M_n ; \Z)$. 
Reducing the matrix by the elementary operations on presentation matrices of modules 
(for example see \cite[Lemma 7.2.1]{KawauchiSurvey}), 
we see that the previous matrix is equivalent to 
\[\begin{pmatrix} 
n^3 - n^2 + n -1 & 0
\end{pmatrix}. \] 
Thus we have $H_1(M_n ; \Z) \cong \Z \oplus \Z_{n^3 - n^2 + n -1}$. 
\end{proof}

\begin{figure}[!htb]
\centering
\begin{overpic}[width=.4\textwidth]{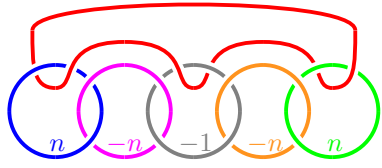}
\put(12,3){$\Blue{n}$} 
\put(27,3){$\Magenta{-n}$} 
\put(46,3){$\Gray{-1}$} 
\put(64,3){$\Orange{-n}$} 
\put(85,3){$\Green{n}$} 
\end{overpic}
\caption{A surgery description of $M_n$}
\label{fig_Mn}
\end{figure}
\begin{figure}[!htb]
\centering
\begin{overpic}[width=.4\textwidth]{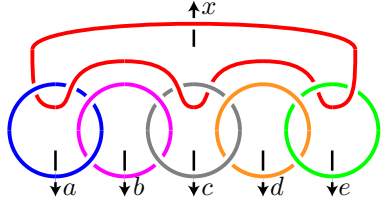}
\put(52,48){$x$}
\put(15.5,.5){$a$}
\put(34,.5){$b$}
\put(52,.5){$c$}
\put(70,.5){$d$}
\put(88,.5){$e$}
\end{overpic}
\caption{Generators of $H_1(X;\Z)$.}
\label{fig_generator}
\end{figure}

\begin{remark} 
By the classification of lens spaces, for example see \cite[Theorem D.1.2]{KawauchiSurvey}, 
$\Sigma_n$ is chiral (i.e.\ not achiral) since we have 
\[ (n^3 - 2n^2 + n -1)^2 \equiv 1 \not\equiv -1  \mod (n^4 - 2n^3 + 2n^2 - 2n +1)\, . \]  
%Therefore $M_n$ is not homeomorphic to any null-homologous knot complement. 
\end{remark}

%In this section, we observe that the homology group of $M_n$ guarantees that 
%$M_n$ is not homeomorphic to any null-homologous knot complement 
%in any closed oriented 3-manifold. 
\begin{remark} 
For any null-homologous knot $J$ in an oriented closed 3-manifold $Y$, 
we have 
$H_1(Y \setminus J ; \Z) \cong \Z \oplus H_1(Y ; \Z)$. 
Each of our 3-manifolds $M_n$ is the exterior of a knot, 
say $J_n$, in the lens space $\Sigma_n$. 
By Lemma~\ref{lem:homology}, 
one can also see that $J_n$ is not null-homologous in $\Sigma_n$ 
since $H_1(\Sigma_n ; \Z) = \Z_{n^4 - 2n^3 + 2n^2 - 2n +1}$. 
In the case where $n = 2$, 
we have $n^4 - 2n^3 + 2n^2 - 2n +1 = n^3 - n^2 + n -1 = 5$. 
This case corresponds to the chirally cosmetic fillings on the figure-eight sibling 
introduced in Section~\ref{sec:intro}. 
Thus, the knot $J_2$ is also not null-homologous in $\Sigma_2$. 
\end{remark}

\begin{remark} 
As shown in Figure~\ref{fig_HandleSlide}, 
we can see directly that the $0$-surgery on the red component in Figure~\ref{fig_Mn} 
changes $\Sigma_{n}$ to $-\Sigma_{n}$. 
\end{remark}

\begin{figure}[!htb]
\centering
\begin{overpic}[width=\textwidth]{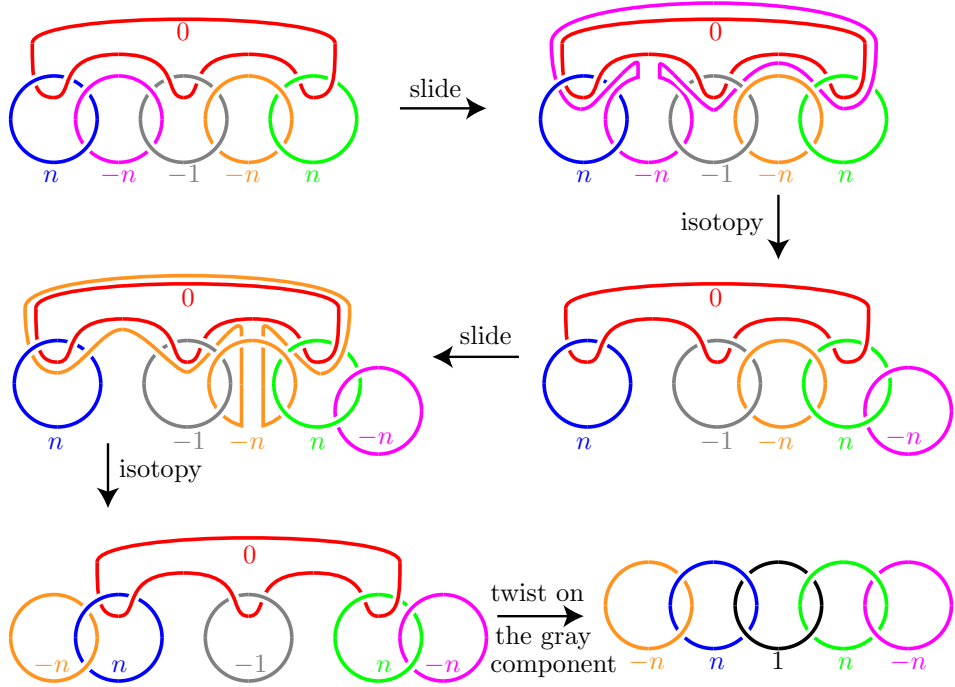}
\put(18,68){$\Red{0}$} 
\put(4,53){$\Blue{n}$} 
\put(10,53){$\Magenta{-n}$} 
\put(17,53){$\Gray{-1}$} 
\put(23,53){$\Orange{-n}$} 
\put(31.5,53){$\Green{n}$} 
\put(42.5,62){slide} 
\put(74,68){$\Red{0}$} 
\put(60,53){$\Blue{n}$} 
\put(66,53){$\Magenta{-n}$} 
\put(73,53){$\Gray{-1}$} 
\put(79,53){$\Orange{-n}$} 
\put(87.5,53){$\Green{n}$} 
\put(71,48){isotopy} 
\put(74,40){$\Red{0}$} 
\put(60,25){$\Blue{n}$} 
\put(92.5,25.5){$\Magenta{-n}$} 
\put(73,25){$\Gray{-1}$} 
\put(79,25){$\Orange{-n}$} 
\put(87.5,25){$\Green{n}$} 
\put(47.8,36){slide} 
\put(18.5,40){$\Red{0}$} 
\put(4.5,25){$\Blue{n}$} 
\put(37,25.5){$\Magenta{-n}$} 
\put(17.5,25){$\Gray{-1}$} 
\put(23.5,25){$\Orange{-n}$} 
\put(32,25){$\Green{n}$} 
\put(12,22){isotopy} 
\put(25,13){$\Red{0}$} 
\put(11.2,1.5){$\Blue{n}$} 
\put(44,1.5){$\Magenta{-n}$} 
\put(24,1.5){$\Gray{-1}$} 
\put(3,1.5){$\Orange{-n}$} 
\put(39,1.5){$\Green{n}$} 
\put(51,9){twist on} 
\put(51.5,4.5){the gray} 
\put(51,1.5){component} 
\put(65.5,1.8){$\Orange{-n}$} 
\put(74,1.8){$\Blue{n}$} 
\put(80.5,1.8){$1$} 
\put(87.5,1.8) {$\Green{n}$} 
\put(93,1.8){$\Magenta{-n}$}
\end{overpic}
\caption{The $0$-surgery changes $\Sigma_{n}$ to $-\Sigma_{n}$. }
\label{fig_HandleSlide}
\end{figure}

%Since $M_n$ admits an orientation-reversing self-homeomorphism which sends 
%the slope $1/0$ to $0/1$, as discussed in Section~\ref{sec:intro}, 
%$M_n$ is not homeomorphic to any amphicheiral null-homologous knot complement 
%in a closed $3$-manifold. 
%Thus we have the following. 

As in the proof of Theorem~\ref{thm:main}, 
$M_n$ admits chirally cosmetic Dehn fillings along distance one slopes. 
Thus, we have the following. 

\begin{corollary}\label{cor:main}
There exist infinitely many achiral 1-cusped hyperbolic 3-manifolds each of which admits chirally cosmetic Dehn fillings along distance one slopes. 
\end{corollary}

%\section{Calculation of homology}\label{sec:homology}

%Here we note that 
%calculating the homology group also guarantees that 
%$M_n$ is not homeomorphic to any null-homologous knot complement 
%in any closed oriented 3-manifold. 

\section{Realizing $M_n$ as an amphicheiral knot complement}\label{sec:amphi}

In this section, we show the following. 

\begin{proposition}\label{prop:realize}
Each of our 3-manifolds $M_n$ can be realized as 
the exterior of an amphicheiral knot in some achiral 3-manifold.  
\end{proposition}

Note that such an amphicheiral knot is not null-homologous. 

\begin{proof}[Proof of Proposition~\ref{prop:realize}] 
As in Sections~\ref{sec:examples} and \ref{sec:hyp}, 
the meridian on $\partial M_n$ is $\widetilde \mu$, 
and we can choose the longitude on $\partial M_n$ as $\widetilde \lambda$ 
corresponding to $0/1$. 
With respect to this meridian-longitude system, 
we see that the two slopes corresponding to $1/1$ and $-1/1$ are invariant 
via the orientation-reversing self-homeomorphism $h = \widetilde r \circ \widetilde m$ on $M_n$, 
see Figure~\ref{fig_InvSlope}. 

Let $N_{n}^{\pm}$ be the closed oriented 3-manifold 
obtained from $M_{n}$ by $\pm 1/1$-Dehn filling respectively. 
Since the core curve $c$ of the attached solid torus $V$ 
intersects a meridian disk of $V$ once, 
$h$ extends to $V$, and thus, $h$ extends to $N^{\pm}_{n}$. 
This implies that $N^{\pm}_{n}$ is achiral. 
Further, for the knot $k$ in $N^{\pm}_{n}$, which corresponds to the core $c$, 
$k$ is an amphicheiral knot in $N^{\pm}_{n}$ and 
the exterior of $k$ is homeomorphic to $M_{n}$. 
\end{proof}

\begin{figure}[!htb]
\centering
\begin{overpic}[width=\textwidth]{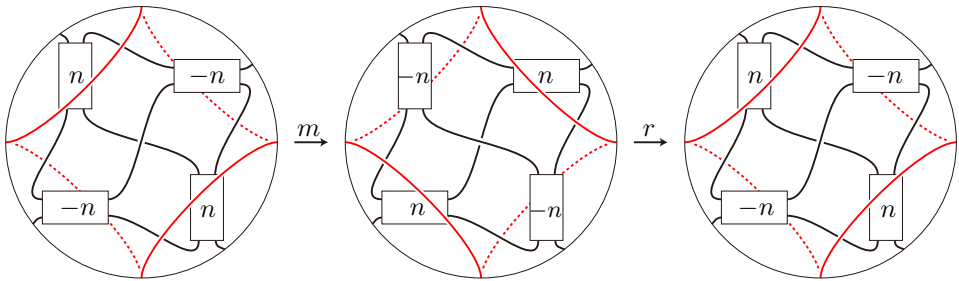}
\put(6.7,20.5){$n$} 
\put(5.7,6.7){$-n$} 
\put(20.5,6.5){$n$} 
\put(19.3,20.5){$-n$} 
\put(41.1,20.5){{\small $-n$}} 
\put(42.5,6.7){$n$} 
\put(55,6.5){{\small $-n$}} 
\put(56,20.5){$n$} 
\put(78,20.5){$n$} 
\put(77,6.7){$-n$} 
\put(92,6.5){$n$} 
\put(90.5,20.5){$-n$} 
\put(30.7,15){$m$} 
\put(67,15){$r$} 
\end{overpic}
\caption{The red loop corresponding to the slope $1/1$ on $\partial T_{n}$ is 
invariant via $r \circ m$.} 
\label{fig_InvSlope}
\end{figure}

\subsection*{Acknowledgement}
The authors would like to thank Professor Kimihiko Motegi and Professor Kai Ishihara 
for useful comments on Section~\ref{sec:amphi}.


\begin{thebibliography}{99}
%identified by Mref : http://www.ams.org/mathscinet-mref

\bibitem{Bleiler}
S. A. Bleiler, 
Banding, twisted ribbon knots, and producing reducible manifolds via Dehn surgery, 
Math. Ann. {\bf 286} (1990), no.~4, 679--696. %MR1045396

\bibitem{BleilerHodgsonWeeks}
S. A. Bleiler, C. D. Hodgson\ and\ J. R. Weeks, 
Cosmetic surgery on knots, 
in {\it Proceedings of the Kirbyfest (Berkeley, CA, 1998)}, 23--34, Geom. Topol. Monogr., 2, Geom. Topol. Publ., Coventry. %MR1734400

\bibitem{HosteNakanishiTaniyama}
J. Hoste, Y. Nakanishi\ and\ K. Taniyama, 
Unknotting operations involving trivial tangles, 
Osaka J. Math. {\bf 27} (1990), no.~3, 555--566. %MR1075165

\bibitem{IJM}
K. Ichihara, I. D. Jong, and H. Masai, 
Cosmetic banding on knots and links, 
to appear in Osaka J. Math., 
arXiv:1602.01542.

\bibitem{IchiharaSaito}
K. Ichihara\ and\ T. Saito, 
Cosmetic surgery and the $SL(2,\C)$ Casson invariant for two-bridge knots, 
to appear in Hiroshima Math. J., 
arXiv:1602.02371.

\bibitem{IchiharaWu}
K. Ichihara and Z. Wu, 
A note on Jones polynomial and cosmetic surgery, 
to appear in Comm. Anal. Geom., 
arXiv:1606.03372. 

\bibitem{KawauchiSurvey}
A. Kawauchi, 
{\it A survey of knot theory}, 
translated and revised from the 1990 Japanese original by the author, Birkh\"auser Verlag, Basel, 1996. %MR1417494

\bibitem{Kirby} 
Problems in low-dimensional topology, 
in {\it Geometric topology (Athens, GA, 1993)}, 35--473, 
AMS/IP Stud. Adv. Math., 2.2, Amer. Math. Soc., Providence, RI. %MR1470751

\bibitem{MartelliPetronio}
B. Martelli\ and\ C. Petronio, 
Dehn filling of the ``magic'' 3-manifold, 
Comm. Anal. Geom. {\bf 14} (2006), no.~5, 969--1026. %MR2287152

\bibitem{Matignon}
D. Matignon, 
On the knot complement problem for non-hyperbolic knots, 
Topology Appl. {\bf 157} (2010), no.~12, 1900--1925. %MR2646423

\bibitem{Montesinos}
J. M. Montesinos, 
Surgery on links and double branched covers of $S\sp{3}$, 
in {\it Knots, groups, and $3$-manifolds (Papers dedicated to the memory of R. H. Fox)}, 227--259. 
Ann. of Math. Studies, 84, Princeton Univ. Press, Princeton, NJ. %MR0380802

\bibitem{NiWu}
Y. Ni\ and\ Z. Wu, 
Cosmetic surgeries on knots in $S^3$, 
J. Reine Angew. Math. {\bf 706} (2015), 1--17. %MR3393360

\bibitem{NikkuniTaniyama}
R. Nikkuni\ and\ K. Taniyama, 
Symmetries of spatial graphs and Simon invariants, 
Fund. Math., {\bf 205} (2009), no.~3, 219--236. 


\bibitem{Weeks}
J. Weeks, 
Hyperbolic structures on three-manifolds, 
PhD thesis, Princeton University, 1985.

\bibitem{Zekovic}
A. Zekovi\'c, 
Computation of Gordian distances and $H_2$-Gordian distances of knots, 
Yugosl. J. Oper. Res. {\bf 25} (2015), no.~1, 133--152. %MR3331990

\end{thebibliography}
\end{document}